\documentclass{amsart}
\usepackage{amsfonts}

\numberwithin{equation}{section}
\newtheorem{theorem}{Theorem}[section]

\newtheorem{defn}[theorem]{Definition}

\newtheorem{corollary}[theorem]{Corollary}

\newtheorem{definition}[theorem]{Definition}
\newtheorem{example}[theorem]{Example}

\newtheorem{lemma}[theorem]{Lemma}

\newtheorem{proposition}[theorem]{Proposition}
\newtheorem{remark}[theorem]{Remark}

\def\K{\mathbf{K}}

\def\R{\mathbb{R}}

\def\Q{\mathbf{Q}}
\def\P{\mathbf{P}}
\def\mrf{\mathbf{MRF}}

\def\beq{\begin{equation}}
\def\eeq{\end{equation}}
\def\y{\mathbf{y}}
\email{rida.laraki@polytechnique.edu}
\email{lasserre@laas.fr}
\thanks{We would like to thank Bernhard von Stengel and the referees for their comments. The work of J.B. Lasserre was supported by the (French) ANR under
grant NT$05-3-41612$.}

\begin{document}
\title{Semidefinite Programming for Min-Max Problems and
Games}
\author{R. Laraki}
\address{CNRS, Laboratoire d'Econom\'etrie de l'Ecole Polytechnique,
91128 Palaiseau-Cedex, France}
\author{J.B. Lasserre}
\address{LAAS-CNRS and Institute of Mathematics, University of Toulouse,
LAAS, 7 avenue du Colonel Roche, 31077 Toulouse C\'edex 4, France}

\begin{abstract}
We consider two min-max problems: (1) minimizing the
supremum of finitely many rational functions over a compact basic
semi-algebraic set and (2) solving a 2-player zero-sum
polynomial game in randomized strategies with compact basic
semi-algebraic sets of pure strategies. In both problems the optimal value
can be approximated by solving a hierarchy of semidefinite relaxations, in
the spirit of the moment approach developed in Lasserre \cite%
{lasserre1,lasserre2}. This provides a unified approach and a class of
algorithms to compute Nash equilibria and min-max strategies of several
static and dynamic games. Each semidefinite relaxation can be solved in time
which is polynomial in its input size and practice on a sample of  experiments
reveals that few relaxations are needed for a good approximation
(and sometimes even for finite convergence), a behavior similar to what was observed in
polynomial optimization.
\end{abstract}

\maketitle

%\subjclass{91A05, 91A06, 90C22, 90-08, 49L99, 47A50 46B42 46H10}date{}

\section{Introduction}
Initially, this paper was motivated by developing a unified methodology for
solving several types of (neither necessarily finite nor zero-sum) $N$-player games.
But it is also of self-interest in optimization for minimizing the maximum of finitely many rational functions
(whence min-max) on a compact basic semi-algebraic set.
Briefly, the moment-s.o.s. approach developed in \cite{lasserre1,lasserre2} is extended
to two large classes of min-max problems. This allows
to obtain a new numerical scheme based on semidefinite programming to
compute approximately and sometimes exactly (1) Nash equilibria and the
min-max of any finite game\footnote{%
Games with finitely many players where the set of pure actions of each
player is finite.} and (2) the value and the optimal strategies of any
polynomial two-player zero-sum game\footnote{%
Games with two players where the set of pure actions of each
player is a compact basic semi-algebraic-set and the payoff function is
polynomial.}. In particular, the approach can be applied to the so-called
Loomis games defined in \cite{Loomis} and to some dynamic games described in Kolhberg \cite{Kohlberg}.

\subsection*{Background.}

Nash equilibrium \cite{Nash 1} is a central concept in non-cooperative game
theory. It is a profile of mixed strategies (a strategy for each player)
such that each player is best-responding to the strategies of the opponents.
An important problem is to compute numerically a Nash equilibrium, an approximate Nash equilibrium for a given precision or all
Nash equilibria in mixed strategies of a finite game.

It is well known that any two-player \textit{zero-sum finite} game (in mixed strategies) is
reducible to a linear program (Dantzig \cite{Dantzig}) and hence equilibria could be
computed in polynomial time (Khachiyan \cite{Khachiyan}).

Lemke and Howson \cite{Lemke} provided a famous algorithm that computes a
Nash-equilibrium (in mixed strategies) of any $2$-player \textit{non-zero-sum} finite game. The
algorithm has been extended to $N$-player finite games by Rosenm\"{u}ller
\cite{Rosenmuller}, Wilson \cite{Wilson} and Govindan and Wilson \cite%
{GovindanWilson}.

An alternative to the Lemke-Howson algorithm for 2-player games is provided
in van den Elzen and Talman \cite{ElzenTalman} and has been extended to $n$%
-player games by Herings and van den Elzen \cite{HeringsElzen2002}. As shown
in the recent survey of Herings and Peeters \cite{HeringsPeeters2009}, all
these algorithms (including the Lemke-Howson) are homotopy-based and converge (only) when the game is non-degenerate.

Recently, Savani and von Stengel \cite{Savani} proved that the Lemke-Howson
algorithm for 2-player games may be exponential. One may expect that this
result extends to all known homotopy methods. Daskalakis, Goldberg and
Papadimitriou \cite{DaskalakisGoldbergPapa} proved that solving numerically
3-player finite games is hard\footnote{%
More precisely, it is complete in the PPAD class of all search problems that
are guaranteed to exist by means of a direct graph argument. This class was
introduced by Papadimitriou in \cite{Papadimitriou} and is between $P$ and $NP$.%
}. The result has been extended to 2-player finite games by Chen and Deng
\cite{DengChen2006}. Hence, computing a Nash equilibrium is as hard as finding a Brouwer-fixed point. For a recent and deep survey on the complexity of Nash equilibria see Papadimitriou \cite{Papadimitriou1}. For the complexity of computing equilibria on game theory in general, see \cite{Nisan} and \cite{Roughgarden2009}.

A different approach to solve the problem is to view the set of Nash
equilibria as the set of real nonnegative solutions to a system of
polynomial equations. Methods of computational algebra (e.g. using Gr\"{o}%
bner bases) can be applied as suggested and studied in e.g. Datta \cite%
{datta}, Lipton \cite{lipton} and Sturmfels \cite{bernd}. However, in this
approach, one first computes \textit{all} complex solutions to sort out all
real nonnegative solutions afterwards. Interestingly, polynomial equations
can also be solved via homotopy-based methods (see e.g. Verschelde \cite{Ve}).

Another important concept in game theory is the min-max payoff of some player $i$, $\underline{v}%
_{i}$. This is the level at which the team of players (other than $i$) can punish player $i$.
The concept plays an important role in repeated games and the famous folk theorem of Aumann and Shapley \cite{AumannShapley}.
To our knowledge, no algorithm in the literature deals with this problem. However, it has been proved recently that computing the min-max for 3 or more player games is NP-hard \cite{Borgs}.

The algorithms described above concern finite games. In the class of
polynomial games introduced by Dresher, Karlin and Shapley (1950), the set
of pure strategies $S^{i}$ of player $i$ is a product of real intervals and
the payoff function of each player is polynomial. When the game is zero-sum
and $S^{i}=[0,1]$ for each player $i=1,2$, Parrilo \cite{Parrilo} showed that
finding an optimal solution is equivalent to solving a single semidefinite
program. In the same framework but with several players, Stein, Parrilo and Ozdaglar \cite{SteinParriloOzdaglar} propose several algorithms to compute correlated equilibria,
one among them using SDP relaxations.  Shah and Parrilo \cite{Shah} extended the methodology in \cite{Parrilo} to discounted zero-sum stochastic games in which the transition is controlled by one of
the two players. Finally, it is worth noticing recent algorithms designed to
solve some specific classes of infinite games (not necessarily polynomial). For instance, G\"{u}rkan and
Pang \cite{Gurkan}.

\subsection*{Contribution}

In the first part we consider what we call the $\mathbf{MRF}$ problem
which consists of minimizing the supremum of finitely many \textit{rational functions} over a
compact basic semi-algebraic set. In the spirit of the moment approach
developed in Lasserre \cite{lasserre1,lasserre2} for polynomial
optimization, we define a hierarchy of semidefinite relaxations (in short
SDP relaxations). Each SDP relaxation is a semidefinite program
which, up to arbitrary (but fixed) precision,
can be solved in polynomial time
and the monotone sequence of optimal
values associated with the hierarchy converges to the optimal value of $%
\mathbf{MRF}$. Sometimes the convergence is finite and a sufficient
condition permits to detect whether a certain relaxation in the hierarchy is
exact (i.e. provides the optimal value), and to extract optimal solutions.
It is shown that computing the min-max or a Nash equilibrium in mixed
strategies for static finite games or dynamic finite absorbing games reduces
to solving an $\mathbf{MRF}$ problem.
For zero-sum finite games in mixed strategies
the hierarchy of SDP relaxations for the associated {\bf MRF} reduces to the first one of the hierarchy, which in turn reduces to a linear program. This
is in support of the claim that the above {\bf MRF} formulation is a natural extension
to the non linear case of the well-known LP-approach \cite{Dantzig} as it reduces to the latter
for finite zero-sum games. In
addition, if the SDP solver uses primal-dual interior points methods and if
the convergence is finite then the algorithm returns \textit{all} Nash
equilibria (if of course there are finitely many).

To compute all Nash equilibria, homotopy algorithms are developed in Kostreva
and Kinard \cite{KostrevaKinard} and Herings and Peeters \cite%
{HeringsPeeters2005}. They apply numerical techniques to obtain all solutions of a
system of polynomial equations.

In the second part, we consider general $2$-player zero-sum polynomial games in mixed strategies
(whose action sets are basic compact semi-algebraic sets of $\mathbb{R}^{n}$
and payoff functions are polynomials). We show that the common value of
max-min and min-max problems can be
approximated as closely as desired, again by solving a certain hierarchy of
SDP relaxations. Moreover, if a certain rank condition is satisfied
at an optimal solution of some relaxation of the hierarchy, then this relaxation is exact
and one may extract optimal strategies.
Interestingly and not surprisingly, as this hierarchy is
derived from a min-max problem over sets of measures, it is a subtle
combination of moment and sums of squares (s.o.s.) constraints whereas the hierarchy for
polynomial optimization is based on either moments (primal formulation)
or s.o.s. (dual formulation) but not both. This is a multivariate extension of
Parrilo's \cite{Parrilo} result for the univariate case where one needs to
solve a single semidefinite program (as opposed to a hierarchy). The
approach may be extended to dynamic absorbing games\footnote{%
In dynamic absorbing games, transitions are controlled by \textit{both}
players, in contrast with Parrilo and Shah \cite{Shah} where only one player
controls the transitions.} with discounted rewards, and in the univariate
case, one can construct a polynomial time algorithm that combines a
binary search with a semidefinite program.

Hence the first main contribution is to
formulate several game problems as a particular instance of the $\mathbf{MRF%
}$ problem while the second main contribution extends the moment-s.o.s. approach of \cite%
{lasserre1,lasserre2} in two directions. The first extension (
the {\bf MRF} problem) is essentially a non trivial adaptation of Lasserre's
approach \cite{lasserre1} to the problem of minimizing
the sup of finitely many rational functions. Notice that the sup of finitely many rational functions is {\it not} a rational function. However one reduces
the initial problem to that of minimizing a {\it single} rational function
(but now of $n+1$ variables) on an appropriate set of $\R^{n+1}$.
As such, this can also be viewed as an
extension of Jibetean and De Klerk's result \cite{jibetean}
for minimizing a single rational function.
The second extension generalizes (to the multivariate case)
Parrilo's approach \cite{Parrilo} for the univariate case,
and provides a hierarchy of mixed moment-s.o.s. SDP relaxations.
The proof of convergence is delicate as
one has to consider simultaneously moment constraints as well
as s.o.s.-representation of positive polynomials. (In particular, and in contrast to
polynomial optimization, the converging
sequence of optimal values associated with the hierarchy of SDP relaxations
is {\em not} monotone anymore.)

To conclude, within the game theory community the rather negative
computational complexity results (\cite{Borgs}, \cite{DaskalakisGoldbergPapa}, \cite{DengChen2006}, \cite{Roughgarden2009}, \cite{Savani}) have
reinforced the idea that solving a game is computationally hard. On a more
positive tone our contribution provides a unified semidefinite programming
approach to many game problems. It shows that optimal value and optimal strategies
can be approximated as closely as desired (and sometimes obtained exactly)
by solving a hierarchy of semidefinite relaxations, very much in the spirit
of the moment approach described in \cite{lasserre1} for solving polynomial
minimization problems (a particular instance of the Generalized Problem of
Moments \cite{lasserre2}). Moreover, the methodology is consistent with
previous results of \cite{Dantzig} and \cite{Parrilo} as it reduces to a linear
program for finite zero-sum games and to a single semidefinite program for
univariate polynomial zero-sum games.

Finally, even if practice from problems in polynomial optimization seems to reveal that this
approach is efficient, of course the size of the semidefinite relaxations
grows rapidly with the initial problem size. Therefore, in view of the
present status of public SDP solvers available, its application is limited
to small and medium size problems so far.

Quoting Papadimitriou \cite{Papadimitriou1}: ``\emph{The PPAD-completeness of Nash suggests that any approach to finding Nash equilibria that aspires to be efficient [...] should explicitly take advantage of computationally beneficial special properties of the game in hand}''.
Hence to make our algorithm efficient for larger size problems, one could exploit possible sparsity and regularities often present in the data (which will be the case if the game is the normal form of an extensive form game). Indeed specific SDP relaxations for minimization problems that exploit sparsity efficiently have been provided in Kojima et al. \cite{waki} and their convergence has been proved in \cite{sparse} under some condition on the sparsity pattern.

\section{Notation and preliminary results}

\subsection{Notation and definitions}

Let $\mathbb{R}[x]$ be the ring of real polynomials in the variables $%
x=(x_1,\ldots,x_n)$ and let $(X^\alpha)_{\alpha\in\mathbb{N}}$ be its
canonical basis of monomials. Denote by $\Sigma[x]\subset\mathbb{R}[x]$ the
subset (cone) of polynomials that are sums of squares (s.o.s.), and by $%
\mathbb{R}[x]_d$ the space of polynomials of degree at most $d$. Finally let
$\Vert x\Vert$ denote the Euclidean norm of $x\in\R^n$.

With $\mathbf{y}=:(y_\alpha)\subset\mathbb{R}$ being a sequence indexed in
the canonical monomial basis $(X^\alpha)$, let $L_\mathbf{y}:\mathbb{R}[x]\to%
\mathbb{R}$ be the linear functional
\begin{equation*}
f\:(=\sum_{\alpha\in\mathbb{N}^n}f_\alpha\,x^\alpha)\longmapsto\:
\sum_{\alpha\in\mathbb{N}^n}f_\alpha\,y_\alpha,\quad f\in\mathbb{R}[x].
\end{equation*}

\subsection*{Moment matrix}

Given $\mathbf{y}=(y_\alpha)\subset\mathbb{R}$, the \textit{moment} matrix $%
M_d(\mathbf{y})$ of order $d$ associated with $\mathbf{y}$, has its rows and
columns indexed by $(x^\alpha)$ and its $(\alpha,\beta)$-entry is defined
by:
\begin{equation*}
M_d(\mathbf{y})(\alpha,\beta)\,:=\,L_\mathbf{y}(x^{\alpha+\beta})\,=\,y_{%
\alpha+\beta},\qquad\vert\alpha\vert,\,\vert\beta\vert\,\leq d.
\end{equation*}

\subsection*{Localizing matrix}

Similarly, given $\mathbf{y}=(y_\alpha)\subset\mathbb{R}$ and $\theta\in%
\mathbb{R}[x] \,(=\sum_\gamma\theta_\gamma x^\gamma$), the \textit{localizing%
} matrix $M_d(\theta, \mathbf{y})$ of order $d$ associated with $\mathbf{y}$
and $\theta$, has its rows and columns indexed by $(x^\alpha)$ and its $%
(\alpha,\beta)$-entry is defined by:
\begin{equation*}
M_d(\theta ,\mathbf{y})(\alpha,\beta):=L_\mathbf{y}(x^{\alpha+\beta}%
\theta(x))=
\sum_{\gamma}\theta_\gamma y_{\gamma+\alpha+\beta},\qquad\vert\alpha\vert,%
\vert\beta\vert\,\leq d.
\end{equation*}
One says that $\mathbf{y}=(y_\alpha)\subset\mathbb{R}$ has a \textit{%
representing} measure supported on $\mathbf{K}$ if there is some finite
Borel measure $\mu$ on $\mathbf{K}$ such that
\begin{equation*}
y_\alpha\,=\,\int_\mathbf{K} x^\alpha\,d\mu(x),\qquad \forall\,\alpha\in%
\mathbb{N}^n.
\end{equation*}
For later use, write
\begin{eqnarray}  \label{local}
M_d(\mathbf{y})&=&\sum_{\alpha\in\mathbb{N}^n}y_\alpha\,B_\alpha \\
M_d(\theta,\mathbf{y})&=&\sum_{\alpha\in\mathbb{N}^n}y_\alpha\,B^{\theta}_%
\alpha,
\end{eqnarray}
for real symmetric matrices $(B_\alpha,B^{\theta}_\alpha)$ of appropriate
dimensions. Note that the above two summations contain only finitely many terms.

\begin{defn}[Putinar's property]
\label{defput} \textrm{Let $(g_{j})_{j=1}^{m}\subset \mathbb{R}[x]$. A basic
closed semi algebraic set $\mathbf{K}:=\{x\in \mathbb{R}^{n}:g_{j}(x)\geq
0,:j=1,\ldots ,m\}$ satisfies Putinar's property if there exists $u\in
\mathbb{R}[x]$ such that $\{x:u(x)\geq 0\}$ is compact and
\begin{equation}
u\,=\,\sigma _{0}+\sum_{j=1}^{m}\sigma _{j}\,g_{j}  \label{putrep}
\end{equation}%
for some s.o.s. polynomials $(\sigma _{j})_{j=0}^{m}\subset \Sigma \lbrack x]$. Equivalently,
for some $M>0$ the quadratic polynomial $x\mapsto M-\Vert x\Vert ^{2}$ has
Putinar's representation (\ref{putrep}).}
\end{defn}

Obviously Putinar's property implies compactness of $\mathbf{K}$. However,
notice that Putinar's property is not geometric but algebraic as it is
related to the representation of $\mathbf{K}$ by the defining polynomials $%
(g_{j})$'s. Putinar's property holds if e.g. the level set $\{x:g_{j}(x)\geq
0\}$ is compact for some $j$, or if
all $g_{j}$ are affine and $\mathbf{K}$  is compact (in which case it is a polytope). In case it is not satisfied and if for some known $M>0$%
, $\Vert x\Vert ^{2}\leq M$ whenever $x\in \mathbf{K}$, then it suffices to
add the redundant quadratic constraint $g_{m+1}(x):=M-\Vert x\Vert ^{2}\geq
0 $ to the definition of $\mathbf{K}$. The importance of Putinar's property
stems from the following result:

\begin{theorem}[Putinar \protect\cite{putinar}]
\label{thput} Let $(g_{j})_{j=1}^{m}\subset \mathbb{R}[x]$ and assume that
\begin{equation*}
\mathbf{K}:=\{\,x\in \mathbb{R}^{n}:g_{j}(x)\geq 0,\quad j=1,\ldots ,m\}
\end{equation*}%
satisfies Putinar's property.

\textrm{(a)} Let $f\in \mathbb{R}[x]$ be strictly positive on $\mathbf{K}$. Then $f$
can be written as $u$ in (\ref{putrep}).

\textrm{(b)} Let $\mathbf{y}=(y_{\alpha })$. Then $\mathbf{y}$ has a
representing measure on $\mathbf{K}$ if and only if
\begin{equation}
M_{d}(\mathbf{y})\succeq 0,\quad M_{d}(g_{j},\mathbf{y})\succeq 0,\qquad
j=1,\ldots ,m;\quad d=0,1,\ldots  \label{put-moments}
\end{equation}
\end{theorem}

We also have:

\begin{lemma}
\label{lemm-frac} Let $\mathbf{K}\subset \mathbb{R}^{n}$ be compact and let $%
p,q$ continuous with $q>0$ on $\mathbf{K}$. Let $M(\mathbf{K})$ be the set
of finite Borel measures on $\mathbf{K}$ and let $P(\mathbf{K})\subset M(%
\mathbf{K})$ be its subset of probability measures on $\mathbf{K}$. Then
\begin{eqnarray}
\min_{\mu \in P(\mathbf{K})}\frac{\int_\K p\,d\mu }{\int_\K q\,d\mu }
&=&\min_{\varphi \in M(\mathbf{K})}\{\int_\K p\,d\varphi :\int_\K q\,d\varphi
\,=\,1\} \\
&=&\min_{\mu \in P(\mathbf{K})}\int_\K \frac{p}{q}\,d\mu \,=\,\min_{x\in
\mathbf{K}}\frac{p(x)}{q(x)}
\end{eqnarray}
\end{lemma}

\begin{proof}
Let $\rho ^{\ast }:=\min_{x}\{p(x)/q(x):x\in \mathbf{K}\}$. As $q>0$ on $%
\mathbf{K}$,
\begin{equation*}
\frac{\int_\K p\,d\mu }{\int_\K q\,d\mu }\,= \,\frac{\int_\K (p/q)\,q\,d\mu }{\int_\K
q\,d\mu }\,\geq \rho ^{\ast }.
\end{equation*}%
Hence if $\mu\in P(\K)$ then $\int_\K (p/q)d\mu \geq \rho ^{\ast }\int_\K d\mu =\rho ^{\ast }$.
On the other hand, with $x^{\ast }\in \mathbf{K}$ a global minimizer of $p/q$ on $%
\mathbf{K}$, let $\mu :=\delta _{x^{\ast }}\in P(\mathbf{K})$ be the Dirac
measure at $x=x^{\ast }$. Then $\int_\K pd\mu /\int_\K qd\mu =p(x^{\ast
})/q(x^{\ast })=\int_\K (p/q)d\mu =\rho ^{\ast }$, and therefore
\begin{equation*}
\min_{\mu \in P(\mathbf{K})}\frac{\int_\K pd\mu }{\int_\K qd\mu }\,=\,\min_{\mu
\in P(\mathbf{K})}\int_\K \frac{p}{q}d\mu \,=\,\min_{x\in \mathbf{K}}:\frac{p(x)%
}{q(x)}=\rho ^{\ast }.
\end{equation*}%
Next, for every $\varphi \in M(\mathbf{K})$ with $\int_\K qd\varphi=1$, $\int_\K
p\,d\varphi \geq \int_\K \rho ^{\ast }\,q\,d\varphi \,=\,\rho ^{\ast }$, and so
$\min_{\varphi \in M(\mathbf{K})}\{\int_\K p\,d\varphi :\int_\K q\,d\varphi
\,=\,1\}\geq \rho ^{\ast }$. Finally taking $\varphi :=q(x^{\ast
})^{-1}\delta _{x^{\ast }}$ yields $\int_\K qd\varphi =1$ and $\int_\K p\,d\varphi
=p(x^{\ast })/q(x^{\ast })=\rho ^{\ast }$.

Another way to see why this is true is throughout the following argument.
The function $\mu \rightarrow \frac{\int_\K p\,d\mu }{\int_\K q\,d\mu }$ is
quasi-concave (and also quasi-convex) so that the optimal value of the
minimization problem is achieved on the boundary.
\end{proof}

\section{Minimizing the max of finitely many rational functions}

Let $\mathbf{K}\subset \mathbb{R}^{n}$ be the basic semi-algebraic set
\begin{equation}
\mathbf{K}\,:=\,\{x\in \mathbb{R}^{n}:g_{j}(x)\geq 0,\quad j=1,\ldots ,p\}
\label{setk}
\end{equation}%
for some polynomials $(g_{j})\subset \mathbb{R}[x]$, and let $%
f_{i}=p_{i}/q_{i}$ be rational functions, $i=0,1,\ldots ,m$, with $%
p_{i},q_{i}\in \mathbb{R}[x]$. We assume that:

$\bullet$  $\K$ satisfies Putinar's property (see Definition \ref{defput}) and,

$\bullet$  $q_i>0$ on $\K$ for every $i=0,\ldots,m$.
\vspace{0.2cm}

Consider the following problem denoted by $\mathbf{MRF}$:
\begin{equation}
\mrf:\quad \rho \,:=\displaystyle\min_{x}\{f_{0}(x)+\max_{i=1,%
\ldots ,m}f_{i}(x):x\in \mathbf{K}\,\},  \label{defp}
\end{equation}%
or, equivalently,
\begin{equation}
\mrf:\quad \rho \,=\displaystyle\min_{x,z}\{f_{0}(x)+z:z\geq
f_{i}(x), \:i=1,\ldots,m;\:x\in \mathbf{K}\,\}.  \label{defpnew}
\end{equation}
With $\mathbf{K}\subset \mathbb{R}^{n}$ as in (\ref{setk}), let $\widehat{%
\mathbf{K}}\subset \mathbb{R}^{n+1}$ be the basic semi algebraic set
\begin{equation}
\widehat{\mathbf{K}}\,:=\,\{(x,z)\in \mathbb{R}^{n}\times \mathbb{R}:x\in
\mathbf{K};\:z\,q_{i}(x)-p_{i}(x)\geq 0,i=1,\ldots ,m\}  \label{newsetk}
\end{equation}%
and consider the new infinite-dimensional optimization problem
\begin{equation}
\mathcal{P}:\quad \hat{\rho}\,:=\,\min_{\mu }\{\int_\K (p_{0}+z\,q_{0})\,d\mu
:\int_\K q_{0}\,d\mu =1,\quad\mu \in M(\widehat{\mathbf{K}})\}  \label{newpb}
\end{equation}%
where $M(\widehat{\mathbf{K}})$ is the set of finite Borel measures on $%
\widehat{\mathbf{K}}$. Problem (\ref{newpb}) is a particular instance of the {\it Generalized Problem of Moments} for which a general methodology (based on a hierarchy of semidefinite relaxations) has been described in \cite{lasserre2}. To make the paper self-contained we explain below how to apply this methodology in the above specific context.

\begin{proposition}
\label{prop1} Let $\mathbf{K}\subset \mathbb{R}^{n}$ be as in (\ref{setk}).
Then $\rho =\hat{\rho}$.
\end{proposition}

\begin{proof}
The following upper and lower bounds
\[\underline{z}:=\displaystyle\min_{i=1,\ldots,m}\: \min_{x\in\K}f_i(x);\quad
\overline{z}:=\displaystyle\max_{i=1,\ldots,m}\: \max_{x\in\K}f_i(x),\]
are both well-defined since $\K$ is compact and $q_i>0$ on $\K$ for every $i=1,\ldots,m$.
Including the additional constraint $\underline{z}\leq z\leq \overline{z}$
in the definition of $\widehat{\K}$ makes it compact without changing the value of
$\rho$. Next observe that
\begin{equation}
\label{aux1}
\rho=\min_{(x,z)}\:\left\{ \frac{p_0(x)+zq_0(x)}{q_0(x)}\,:\:(x,z)\in\widehat{\K}\:\right\}.
\end{equation}
Applying Lemma \ref{lemm-frac} with $\widehat{\K}$ in lieu of $\K$,
and with $(x,z)\mapsto p(x,z):=p_0(x)+zq_0(x)$ and
$(x,z)\mapsto q(x,z):=q_0(x)$, yields the desired result.
\end{proof}

\begin{remark}
\label{remark1}
{\rm
(a) When $m=0$, that is when one wishes to minimize the rational function $f_0$ on $\K$,
then $\mathcal{P}$ reads $\min_\mu \{\int_\K f_0d\mu\,:\,\int_\K q_0d\mu=1;\,\mu\in M(\K)\}$.
Using the dual problem $\mathcal{P}^*:\max \,\{z : p_0(x)-zq_0(x) \geq0\quad\forall x\in \K\}$, Jibetean and DeKlerk \cite{jibetean}
proposed to approximate the optimal value by solving a hierarchy
of semidefinite relaxations. The case $m\geq 2$ is a nontrivial extension of the
case $m=0$ because one now wishes to minimize on $\K$
a function which is {\it not} a rational function. However,
by adding an extra variable $z$, one obtains the moment problem (\ref{newpb}), which indeed is the same as
minimizing the rational function $(x,z)\mapsto  \frac{p_0(x)+z\,q_0(x)}{q_0(x)}$
on a domain $\widehat{\K}\subset\R^{n+1}$. And the dual of $\mathcal{P}$ now reads
$\max\, \{\rho : p_0(x)+zq_0(x)-\rho q_0(x) \geq0\quad\forall (x,z)\in \widehat{\K}\}$. Hence if $\widehat{\K}$ is compact and satisfies Putinar's property
one may use the hierarchy of semidefinite relaxations defined in \cite{jibetean} and adapted to this specific context; see also \cite{lasserre2}.

(b) The case $m=1$ (i.e., when one wants to minimize the sum of rational functions $f_0+f_1$ on $\K$) is also interesting. One way is to reduce to the same denominator
$q_0\times q_1$ and minimize the rational function $(p_0q_1+p_1q_0)/q_0q_1$. But then
the first SDP relaxation of \cite{lasserre2,jibetean} would have to consider polynomials of degree at least $d:=\max[{\rm deg}\,p_0+{\rm deg}\,q_1,
{\rm deg}\,p_1+{\rm deg}\,q_0,{\rm deg}\,q_0+{\rm deg}\,q_1]$, which may be very penalizing
when $q_0$ and $q_1$ have large degree (and sometimes it may even be impossible!).
Indeed this SDP relaxation has as many as $O(n^{d})$ variables and a linear matrix inequality of
size $O(n^{\lceil d/2\rceil})$. In contrast, by proceeding as above in introducing the additional variable $z$, one now minimizes the rational function $(p_0+zq_0)/q_0$ which may be highly preferable
since the first relaxation only considers polynomials of degree bounded by
$\max[{\rm deg}\,p_0,1+{\rm deg}\,q_0]$ (but now in $n+1$ variables). For instance, if
${\rm deg}\,p_i={\rm deg}\,q_i=v$, $i=1,2$, then one has $O(n^{v+1})$ variables instead of $O(n^{2v})$
variables in the former approach.
}
\end{remark}
We next describe how to solve the $\mrf$ problem via a hierarchy of semidefinite
relaxations.

\subsection*{SDP relaxations for solving the $\mrf$ problem}

As $\mathbf{K}$ is compact and $q_i>0$ on $\K$, for all $i$, let
\begin{equation}
M_{1}\,:=\,\max_{i=1,\ldots ,m}\left\{ \frac{\max \{|p_{i}(x)|,x\in \mathbf{K%
}\}}{\min \{q_{i}(x),x\in \mathbf{K}\}}\right\} ,  \label{boundM1}
\end{equation}%
and
\begin{equation}
M_{2}\,:=\,\min_{i=1,\ldots ,m}\left\{ \frac{\min \{p_{i}(x),x\in \mathbf{K}%
\}}{\max \{q_{i}(x),x\in \mathbf{K}\}}\right\} .  \label{boundM2}
\end{equation}%
Redefine the set $\widehat{\mathbf{K}}$ to be
\begin{equation}
\widehat{\mathbf{K}}\,:=\,\left\{ (x,z)\in \mathbb{R}^{n}\times \mathbb{R}%
:h_{j}(x,z)\,\geq \,0,\quad j=1,\ldots p+m+1\right\}  \label{newsetknew}
\end{equation}%
with
\begin{equation}
\left\{
\begin{array}{lll}
(x,z)\mapsto & h_{j}(x,z)\,:=\,g_{j}(x) & j=1,\ldots ,p \\
(x,z)\mapsto & h_{j}(x,z)\,:=\,z\,q_{j}(x)-p_{j}(x) & j=p+1,\ldots ,p+m \\
(x,z)\mapsto & h_{j}(x,z)\,:=\,(M_{1}-z)(z-M_2) & j=m+p+1
\end{array}%
\right. .  \label{defh}
\end{equation}

\begin{lemma}
Let $\mathbf{K}\subset\mathbb{R}^n$ satisfy Putinar's property.
Then the set $\widehat{\mathbf{K}}\subset\mathbb{R}^{n+1}$
defined in (\ref{newsetknew}) satisfies Putinar's property.
\end{lemma}

\begin{proof}
Since $\mathbf{K}$ satisfies Putinar's property,
equivalently, the quadratic polynomial $x\mapsto u(x):=M-\Vert x\Vert ^{2}$ can be
written in the form (\ref{putrep}), i.e.,
$u(x)=\sigma_0(x)+\sum_{j=1}^p\sigma_j(x)g_j(x)$ for some s.o.s. polynomials $(\sigma_j)\subset\Sigma[x]$. Next,
consider the quadratic polynomial
\begin{equation*}
(x,z)\mapsto w(x,z)\,=\,M-\Vert x\Vert ^{2} +(M_1-z)(z-M_2).
\end{equation*}%
Obviously, its level set $\{x:w(x,z)\geq 0\}\subset \mathbb{R}^{n+1}$ is
compact and moreover, $w$ can be written in the form
\begin{eqnarray*}
w(x,z)&=&\sigma _{0}(x)+\sum_{j=1}^{p}\sigma
_{j}(x)\,g_{j}(x)+(M_1-z)(z-M_2)\\
&=&
\sigma' _{0}(x,z)+\sum_{j=1}^{m+p+1}\sigma' _{j}(x,z)\,h_{j}(x,z)
\end{eqnarray*}
for appropriate s.o.s. polynomials $(\sigma'_j)\subset\Sigma[x,z]$. Therefore $\widehat{\mathbf{K}}$
satisfies Putinar's property in Definition \ref{defput}, the desired result.
\end{proof}

We are now in position to define the hierarchy of semidefinite relaxations
for solving the $\mrf$ problem. Let $\mathbf{y}=(y_\alpha)$ be a real sequence
indexed in the monomial basis $(x^\beta z^k)$ of $\mathbb{R}[x,z]$ (with $%
\alpha=(\beta,k)\in\mathbb{N}^n\times\mathbb{N}$).

Let $h_{0}(x,z):=p_{0}(x)+zq_{0}(x)$, and let $v_{j}:=\lceil(\mathrm{deg}%
\,h_{j})/2\rceil$ for every $j=0,\ldots ,m+p+1$. For $r\geq r_{0}:=%
\displaystyle\max_{j=0,\ldots ,p+m+1}v_{j}$, introduce the hierarchy of semidefinite
programs:
\begin{equation}
\mathbf{Q}_{r}:\left\{
\begin{array}{lll}
\displaystyle\min_{\mathbf{y}} & L_{\mathbf{y}}(h_{0}) &  \\
\mathrm{s.t.} & M_{r}(\mathbf{y}) & \succeq 0 \\
& M_{r-v_{j}}(h_{j},\mathbf{y}) & \succeq 0,\quad j=1,\ldots ,m+p+1 \\
& L_{y}(q_{0}) & =1,%
\end{array}%
\right. \   \label{lower}
\end{equation}%
with optimal value denoted $\inf \mathbf{Q}_{r}$ (and $\min \mathbf{Q}_{r}$
if the infimum is attained).

\begin{theorem}
\label{thmain} Let $\mathbf{K}\subset\mathbb{R}^n$ (compact) be as in (\ref%
{setk}). Let $\mathbf{Q}_r$ be
the semidefinite program (\ref{lower}) with $(h_j)\subset\mathbb{R}[x,z]$
and $M_1,M_2$ defined in (\ref{defh}) and (\ref{boundM1})-(\ref{boundM2})
respectively. Then:

\textrm{(a)} $\inf\mathbf{Q}_r\uparrow \rho$ as $r\to\infty$.

\textrm{(b)} Let $\mathbf{y}^r$ be an optimal solution of the SDP relaxation
$\mathbf{Q}_r$ in (\ref{lower}). If
\begin{equation}  \label{finiteconv}
\mathrm{rank}\,M_r(\mathbf{y}^r)\,=\,\mathrm{rank}\,M_{r-r_0}(\mathbf{y}%
^r)\,=\,t
\end{equation}
then $\min\Q_r=\rho$ and one may extract $t$ points $(x^*(k))_{k=1}^t\subset\mathbf{K}$, all global minimizers
of the $\mrf$ problem.

\textrm{(c)} Let $\mathbf{y}^r$ be a nearly optimal solution of the SDP relaxation (\ref{lower})
(with say $\inf\Q_r\leq L_{\y^r}\leq\inf\Q_r+1/r$. If (\ref{aux1}) has a unique global minimizer $x^*\in\K$
then the vector of first-order moments $(L_{\y^r}(x_1),\ldots,L_{\y^r}(x_n))$ converges
to $x^*$ as $r\to\infty$.
\end{theorem}
\begin{proof}
As already mentioned in Remark \ref{remark1},
convergence of the dual of the semidefinite relaxations (\ref{lower})
was first proved in Jibetean and de Klerk \cite{jibetean} for minimizing a rational function
on a basic compact semi-algebraic set (in our context, for minimizing the rational function $(x,z)\mapsto (p_0(x)+zq_0(x))/q_0(x)$ on the set $\widehat{\K}\subset\R^{n+1}$).
See also \cite[\S 4.1]{lasserre2} and \cite[Theor. 3.2]{lasserre22}.
In particular to get (b) see \cite[Theor. 3.4]{lasserre22}. The proof of
(c) is easily adapted from Schweighofer \cite{markus}.
\end{proof}

\begin{remark}\label{remark-selection}
{\rm Hence, by Theorem \ref{thmain}(b), when finite convergence occurs one may extract
$t:={\rm rank}\,M_r(\y)$ global minimizers.
On the other hand, a {\em generic} $\mrf$ problem has a unique global minimizer $x^*\in\K$
and in this case, even when the convergence is only asymptotic, one may still obtain an
approximation of $x^*$ (as closely as desired) from the vector of first-order moments $(L_{\y^r}(x_1),\ldots,L_{\y^r}(x_n))$. For instance, one way to have a unique global minimizer is to
$\epsilon$-perturb the objective function of the $\mrf$ problem by some randomly generated polynomial of a sufficiently large degree, or
to slightly perturb the coefficients of the data $(h_i,g_j)$ of the $\mrf$ problem.
}
\end{remark}

To solve (\ref{lower}%
) one may use e.g. the Matlab based public software GloptiPoly 3 \cite%
{gloptipoly3} dedicated to solve the generalized problem of moments
described in \cite{lasserre2}. It is an extension of GloptiPoly \cite{acm}
previously dedicated to solve polynomial optimization problems. A procedure
for extracting optimal solutions is implemented in Gloptipoly when the rank
condition (\ref{finiteconv}) is satisfied\footnote{%
In fact GloptiPoly 3 extracts \textit{all} solutions because most
SDP solvers that one may call in GloptiPoly 3 (e.g. SeDuMi) use primal-dual
interior points methods with the self-dual embedding technique
which find an optimal solution in the relative
interior of the set of optimal solutions; see \cite[\S 4.4.1, p. 663]{lasserre3}. In the present context of (\ref{lower}) this
means that at an optimal solution $\mathbf{y}^{*}$, the moment matrix $M_{r}(%
\mathbf{y}^{*})$ has maximum rank and its rank corresponds to the number of
solutions.}. For more details the interested reader is referred to \cite%
{gloptipoly3} and \texttt{www.laas.fr/$\sim$henrion/software/}.
\begin{remark}
\label{remark-linear}
{\rm If $g_j$ is affine for every $j=1,\ldots,p$
and if $p_j$ is affine and $q_j\equiv1$ for every $j=0,\ldots,m$, then $h_j$
is affine for every $j=0,\ldots,m$. One may also replace the single quadratic constraint
$h_{m+p+1}(x,z)=(M_1-z)(z-M_2)\geq0$ with the two equivalent linear constraints
$h_{m+p+1}(x,z)=M_1-z\geq0$ and $h_{m+p+2}(x,z)=z-M_2\geq0$.
In this case,  it suffices to solve the
single semidefinite relaxation $\mathbf{Q}_1$, which is in fact a linear
program. Indeed, for $r=1$, $\mathbf{y}=(y_0,(x,z),Y)$ and
\begin{equation*}
M_1(\mathbf{y})\,=\,\left[%
\begin{array}{ccc}
y_0 & \vert & (x\quad z) \\
- &  & - \\
\left(%
\begin{array}{c}
x \\
z%
\end{array}%
\right) & \vert & Y%
\end{array}%
\right].
\end{equation*}
Then (\ref{lower}) reads
\begin{equation*}
\mathbf{Q}_{1}:\left\{
\begin{array}{lll}
\displaystyle\min_{\mathbf{y}} & h_0(x) &  \\
\mathrm{s.t.} & M_{1}(\mathbf{y}) & \succeq 0 \\
& h_{j}(x,z) & \geq0,\quad j=1,\ldots ,m+p+2 \\
& y_0 & =1.%
\end{array}%
\right..
\end{equation*}
As $v_j=1$ for every $j$, $M_{1-1}(h_j,\mathbf{y})\succeq0%
\Leftrightarrow M_0(h_j, \mathbf{y})=L_\mathbf{y}(h_j)=h_j(x,z)\geq0$, a
linear constraint. Hence the constraint $M_1(\mathbf{y})\succeq0$ can be
discarded as given any $(x,z) $ one may always find $Y$ such that $M_1(%
\mathbf{y})\succeq0$. Therefore, (\ref{lower}) is a linear program. This is fortunate for finite zero-sum games applications since computing the value is equivalent to minimizing a maximum of finitely many linear functions (and it is already known that it can be solved by Linear Programming).
}
\end{remark}

\section{Applications to finite games}

In this section we show that several solution concepts of static and dynamic finite games
reduce to solving the $\mrf$ problem (\ref{defp}). Those are just examples and one expects that such a reduction also holds in a much larger class of games (when they are described by
finitely many scalars).

\subsection{Standard static games}

A finite game is a tuple $(N,\left\{ S^{i}\right\} _{i=1,...,N},\left\{
g^{i}\right\} _{i=1,...,N})$ where $N\in \mathbb{N}$ is the set of players, $%
S^{i}$ is the finite set of pure strategies of player $i$ and $g^{i}:\mathbf{%
S}\rightarrow \mathbb{R}$ is the payoff function of player $i$, where $%
\mathbf{S}:=S^{1}\times ...\times S^{N}.$ The set
\begin{equation*}
\Delta ^{i}=\left\{ \left( p^{i}(s^{i})\right) _{s^{i}\in S^{i}}:\quad
p^{i}(s^{i})\geq 0,\sum_{s^{i}\in S^{i}}p^{i}(s^{i})=1\right\}
\end{equation*}%
of probability distributions over $S^{i}$ is called the set of mixed
strategies of player $i$. Notice that $\Delta ^{i}$ is a compact basic
semi-algebraic set. If each player $j$ chooses the mixed strategy $%
p^{j}(\cdot ),$ the vector denoted $p=\left( p^{1},...,p^{N}\right) \in
\mathbf{\Delta :}=\Delta ^{1}\times ...\times \Delta ^{N}$ is called a
\textit{profile} of mixed strategies and the expected payoff of a player $i$
is

\begin{equation*}
g^{i}(p)=\sum_{s=(s^{1},...,s^{N})\in S}p^{1}(s^{1})\times ...\times
p^{i}(s^{i})\times ...\times p^{N}(s^{N})g^{i}(s).
\end{equation*}

This is nothing but the multi-linear extension of $g^i$. For a player $i$, and a profile $p,$ let $p^{-i}$ be the profile of the
other players except $i$: that is $%
p^{-i}=(p^{1},...,p^{i-1},p^{i+1},...,p^{N}).$ Let $\mathbf{S}%
^{-i}=S^{1}\times ...\times S^{i-1}\times S^{i+1}\times ...\times S^{N}$ and
define
\begin{equation*}
g^{i}(s^{i},p^{-i})=\sum_{s^{-i}\in S^{-i}}
p^{1}(s^{1})\times ...\times p^{i-1}(s^{i-1})\times
p^{i+1}(s^{i+1})\times ...\times p^{N}(s^{N})g^{i}(s),
\end{equation*}
where $s^{-i}:=(s^{1},...,s^{i-1},s^{i+1},...,s^{N})\in S^{-i}$.

A profile $p_{0}$ is a Nash \cite{Nash 1} equilibrium if and only for all $i=1,...,N$ and
all $s^{i}\in S^{i}$, $g^{i}(p_{0})\geq g^{i}(s^{i},p_{0}^{-i})$ or
equivalently if:%
\begin{equation}
\label{iff}
p_{0}\in \arg \min_{p\in \mathbf{\Delta }}\:\left\{\max_{i=1,...,N}\max_{s^{i}\in
S^{i}}\left\{ g^{i}(s^{i},p_{0}^{-i})-g^{i}(p_{0})\right\}\right\} .
\end{equation}

Since each finite game admits at least one Nash equilibrium \cite{Nash 1}, the optimal value of the min-max problem (\ref{iff}) is zero. Notice that
(\ref{iff}) is a particular instance of the {\bf MRF} problem
(\ref{defp}) (with a set $\K=\Delta$ that satisfies Putinar's property and with $q_i=1$ for every $i=0,\ldots,m$), and so Theorem \ref{thmain} applies.
Finally, observe that the number $m$ of polynomials in the inner double "max"
of (\ref{iff}) (or, equivalently, $m$ in (\ref{defp})) is just $m=\sum_{i=1}^n \vert S^i\vert$, i.e.,
$m$ is just the total number of all possible actions.

Hence by solving the hierarchy of
SDP relaxations (\ref{lower}), one can approximate the value of the min-max problem as
closely as desired. In addition, if (\ref{finiteconv}) is satisfied at some
relaxation $\mathbf{Q}_{r}$, then one may extract all the Nash equilibria of the game.

If there is a unique equilibrium $p^*$ then by Theorem \ref{thmain}(c), one may
obtain a solution arbitrary close to $p^*$ and so obtain an $\epsilon$-equilibrium in finite time. Since game problems are not generic $\mrf$ problems, they
have potentially several equilibria which are all  global minimizers of the associated
$\mrf$ problem. Also, perturbing the data of a finite game still leads to
a non generic associated $\mrf$ problem with possibly multiple solutions. However, as in Remark \ref{remark-selection}, one could perturb the $\mrf$ problem
associated with the original game problem to obtain (generically) an
$\epsilon$-perturbed $\mrf$ problem with a unique optimal solution. Notice that the $\epsilon$-perturbed $\mrf$ is not necessarily coming from a finite game. Doing so, by Theorem \ref{thmain}(c), one obtains a sequence that converges asymptotically (and sometimes in finitely many steps)
to an $\epsilon$-equilibrium of the game problem.
%Another important computational issue is to provide a bound on the integer $r_\epsilon$ for a given error $\epsilon$.
Recently, Lipton, Markakis and Mehta \cite{LiptonMarkakisMehta} provided an algorithm that computes an $\epsilon$-equilibrium in less than exponential time but still not polynomial (namely $n^{\frac{logn}{\epsilon^2}}$ where $n$ is the total number of strategies). This promising result yields Papadimitriou \cite{Papadimitriou1} to argue that ``\emph{finding a mixed Nash equilibrium is PPAD-complete raises some interesting questions regarding the usefulness of Nash equilibrium, and helps focus our interest in alternative notions (most interesting among them the approximate Nash equilibrium)}''.

\begin{example}
\label{ex1} {\small {\rm Consider the simple illustrative example of a $%
2\times 2$ game with data
\begin{equation*}
\begin{array}{ccc}
& s_1^2 & s_2^1 \\
s_1^1 & (a,c) & (0,0) \\
s_2^1 & (0,0) & (b,d)%
\end{array}%
\end{equation*}
for some scalars $(a,b,c,d)$. Denote $x_1\in[0,1]$ the probability for player $%
1$ of playing $s_1^1$ and $x_2\in[0,1]$ the probability for player $2$ of
playing $s_1^2$. Then one must solve
\begin{equation*}
\min_{x_1,x_2\in[0,1]}\:\max\left\{%
\begin{array}{l}
ax_1-ax_1x_2-b(1-x_1)(1-x_2) \\
b(1-x_2)-ax_1x_2-b(1-x_1)(1-x_2) \\
cx_1-cx_1x_2-d(1-x_1)(1-x_2) \\
d(1-x_1)-cx_1x_2-d(1-x_1)(1-x_2)%
\end{array}%
\right..
\end{equation*}
We have solved the hierarchy of semidefinite programs (\ref{lower}) with
GloptiPoly 3 \cite{gloptipoly3}. For instance, the moment matrix $M_1(%
\mathbf{y})$ of the first SDP relaxation $\mathbf{Q}_1$ reads
\begin{equation*}
M_1(\mathbf{y})\,=\,\left[%
\begin{array}{cccc}
y_0 & y_{100} & y_{010} & y_{001} \\
y_{100} & y_{200} & y_{010} & y_{001} \\
y_{010} & y_{110} & y_{020} & y_{011} \\
y_{001} & y_{101} & y_{011} & y_{002}\\
\end{array}%
\right],
\end{equation*}
and $\mathbf{Q}_1$ reads
\begin{equation*}
\mathbf{Q}_{1}:\left\{
\begin{array}{ll}
\displaystyle\min_{\mathbf{y}} & y_{001} \\
\mathrm{s.t.} & M_{1}(\mathbf{y})\,\succeq 0 \\
& y_{001}-ay_{100}+ay_{110}+b(y_0-y_{100}-y_{010}+y_{110})\,\geq 0 \\
& y_{001}-by_0+by_{010}+ay_{110}+b(y_0-y_{100}-y_{010}+y_{110})\,\geq 0 \\
& y_{001}-cy_{100}+cy_{110}+d(y_0-y_{100}-y_{010}+y_{110})\,\geq 0 \\
& y_{00}-dy_0+dy_{100}+cy_{110}+d(y_0-y_{100}-y_{010}+y_{110})\,\geq 0 \\
&y_{100}-y_{200}\geq0;\: y_{010}-y_{020}\geq0 ;\: 9-y_{002}\geq0\\
& y_0 \,=1%
\end{array}%
\right..
\end{equation*}
With $(a,b,c,d)=(0.05,0.82,0.56,0.76)$, solving $\mathbf{Q}_3$ yields the
optimal value $3.93.10^{-11}$ and the three optimal solutions $(0,0)$, $%
(1,1) $ and $(0.57575,0.94253)$. With randomly generated $a,b,c,d\in [0,1]$
we also obtained a very good approximation of the global optimum $0$ and $3$
optimal solutions in most cases with $r=3$ (i.e. with moments or order $6$
only) and sometimes $r=4$. }}

{\small {\rm We have also solved 2-player non-zero-sum $p\times q$ games
with randomly generated reward matrices $A,B\in \mathbb{R}^{p\times q}$ and $%
p,q\leq 5$. We could solve $(5,2)$ and $(4,q)$ (with $q\leq 3$) games
exactly with the $4$th (sometimes $3$rd) SDP relaxation, i.e. $\inf \mathbf{Q%
}_{4}=O(10^{-10})\approx 0$ and one extracts an optimal solution \footnote{%
In general, it is not known which relaxation suffices to solve the min-max problem. Also, as already mentioned, GloptiPoly 3 extracts \textit{all} solutions
because most SDP solvers that one may call in GloptiPoly 3 (e.g. SeDuMi) use
primal-dual interior points methods with the self-dual embedding technique which find an optimal solution in the relative interior of the set of optimal solutions. This is explained in \cite[\S 4.4.1, p. 663]{lasserre3}}. However, the size is relatively
large and one is close to the limit of present public SDP solvers like
SeDuMi. Indeed, for a $2$-player $(5,2)$ or $(4,3)$ game, $\mathbf{Q}_{3}$
has 923 variables and $M_{3}(\mathbf{y})\in \mathbb{R}^{84\times 84}$,
whereas $\mathbf{Q}_{4}$ has 3002 variables and $M_{4}(\mathbf{y})\in
\mathbb{R}^{210\times 210}$. For a $(4,4) $ game $\mathbf{Q}_{3}$ has $1715$
variables and $M_{3}(\mathbf{y})\in \mathbb{R}^{120\times 120}$ and $\mathbf{%
Q}_{3}$ is still solvable, whereas $\mathbf{Q}_{4}$ has $6434$ variables and
$M_{4}(\mathbf{y})\in \mathbb{R}^{330\times 330}$.

Finally we have also solved randomly generated instances
of 3-player non-zero sum games with $(2,2,2)$ actions and $(3,3,2)$ actions.
In all $(2,2,2)$ cases the 4th relaxation
$\mathbf{Q}_4$ provided the optimal value and the rank-test
(\ref{finiteconv}) was passed (hence allowing to extract
global minimizers). For the $(3,3,2)$ games,
the third relaxation $\mathbf{Q}_3$ was enough in 30\% of cases
and the fourth relaxation $\mathbf{Q}_4$ in 80\% of cases.

}}
\end{example}

Another important concept in game theory is the min-max payoff $\underline{v}%
_{i}$ which plays an important role in repeated games (Aumann and Shapley \cite%
{AumannShapley}):

\begin{equation*}
\underline{v}_{i}=\min_{p^{-i}\in \mathbf{\Delta }^{-i}}\max_{s^{i}\in
S^{i}}g^{i}(s^{i},p^{-i})
\end{equation*}%
where $\mathbf{\Delta }^{-i}=\Delta ^{1}\times ...\times \Delta ^{i-1}\times
\Delta ^{i+1}\times ...\times \Delta ^{N}.$ This is again a particular
instance of the $\mrf$ problem (\ref{defp}). Hence, it seems more difficult to
compute the approximate min-max strategies compared to approximate Nash equilibrium strategies
because we do not know in advance the value of $\underline{v}_{i}$ while we know that the min-max value associated to the Nash problem is always zero. This is not surprising: in theory, computing a Nash-equilibrim is PPAD-complete \cite{Papadimitriou1} while computing the min-max is NP-hard \cite{Borgs}. In the case of two players, the function $g^{i}(s^{i},p^{-i})$ is linear in $p^{-i}.$ By
remark \ref{remark-linear} it suffices to solve the first relaxation $%
\mathbf{Q}_{1},$ a linear program.

\begin{remark}
{\rm The Nash equilibrium problem may be reduced to solving a system of polynomial
equations (see e.g. \cite{datta}). In the same spirit, an alternative for the Nash-equilibrium
problem (but not for the $\mrf$ problem in general) is to apply the moment approach described in Lasserre et al. \cite{lasserre3} for finding \textit{real}
roots of polynomial equations. If there are finitely many Nash equilibria
then its convergence is {\em finite} and in contrast with the algebraic
methods \cite{datta,lipton,bernd} mentioned above, it provides \textit{all}
real solutions without computing all complex roots.
}
\end{remark}

\subsection{Loomis games}

Loomis \cite{Loomis} extended the min-max theorem to zero-sum games with a rational fraction. His model may be
extended to $N$-player games as follows. Our extension is justified by the next section.

Associated with each player $i\in N$ are two functions $g^{i}:\mathbf{S}\rightarrow \mathbb{R}$ and $f^{i}:\mathbf{S}%
\rightarrow \mathbb{R}$ where $f^{i}>0$ and $%
\mathbf{S}:=S^{1}\times ...\times S^{N}$. With same notation as in the last section, let their multi-linear
extension to $\mathbf{\Delta }$ still be denoted by $g^{i}$ and $f^{i}$. That is, for $p\in\mathbf{\Delta }$, let:
\begin{equation*}
g^{i}(p)=\sum_{s=(s^{1},...,s^{N})\in S}p^{1}(s^{1})\times ...\times
p^{i}(s^{i})\times ...\times p^{N}(s^{N})g^{i}(s).
\end{equation*}
and similarly for $f^i$.

\begin{definition}
\textrm{A Loomis game is defined as follows. The strategy set of player $%
i$ is $\Delta ^{i}$ and
if the profile $p\in \mathbf{\Delta }$ is chosen, his payoff function
is $h^{i}(p)=\frac{g^{i}(p)}{f^{i}(p)}$.}
\end{definition}

One can show the following new lemma\footnote{As far as we know, non-zero sum Loomis games are not considered in the literature. This model could be of interest in situations where there are populations with many players. A mixed strategy for a population describes the proportion of players in the population that uses some pure action. $g^{i}(p)$ is the non-normalized payoff of population $i$ and $f^{i}(p)$ may be interpreted as the value of money for population $i$ so that $h^{i}(p)=\frac{g^{i}(p)}{f^{i}(p)}$ is the normalized payoff of population $i$.}.

\begin{lemma}
A Loomis game admits a Nash equilibrium\footnote{Clearly, the lemma and its proof still hold in infinite games where the sets $S^i$ are convex-compact-metric and the functions $f^i$ and $g^i$ are continuous. The summation in the multi-linear extension should be replaced by an integral.}.
\end{lemma}

\begin{proof}
Note that each payoff function is quasi-concave in $p^{i}$ (and also
quasi-convex so that it is a quasi-linear function). Actually, if $%
h^{i}(p_{1}^{i},p^{-i})\geq \alpha $ and $h^{i}(p_{2}^{i},p^{-i})\geq \alpha
$ then for any $\delta \in \lbrack 0,1]$,%
\begin{equation*}
g^{i}(\delta p_{1}^{i}+(1-\delta )p_{2}^{i},p^{-i})\geq f^{i}(\delta
p_{1}^{i}+(1-\delta )p_{2}^{i},p^{-i})\alpha ,
\end{equation*}
so that $h^{i}(\delta p_{1}^{i}+(1-\delta )p_{2}^{i},p^{-i})\geq \alpha $. \
One can now apply Glicksberg's \cite{Glicksberg} theorem because the
strategy sets are compact, convex, and the payoff functions are
quasi-concave and continuous.
\end{proof}

\begin{corollary}
A profile $p_{0}\in \mathbf{\Delta }$\textbf{\ }is a Nash equilibrium of a Loomis game
if and only if%
\begin{equation}
\label{minmaxloomis}
p_{0}\in \arg \min_{p\in \Delta }\:\left\{\max_{i=1,...,N}\,\max_{s^{i}\in
S^{i}}\left\{ h^{i}(s^{i},p^{-i})-h^{i}(p)\right\}\right\} .
\end{equation}
\end{corollary}

\begin{proof}
Clearly, $p_{0}\in \mathbf{\Delta }$\textbf{\ }is an equilibrium of the
Loomis game if and only if%
\begin{equation*}
p_{0}\in \arg \min_{p\in \Delta }\:\left\{\max_{i=1,...,N}\max_{\widetilde{p}^{i}\in
\Delta ^{i}}\left\{ \frac{g(\widetilde{p}^{i},p^{-i})}{f^{i}(\widetilde{p}%
^{i},p^{-i})}-\frac{g^{i}(p)}{f^{i}(p)}\right\}\right\} .
\end{equation*}%
Using the quasi-linearity of the payoffs or Lemma \ref{lemm-frac}, one
deduces that:
\begin{equation*}
\max_{\widetilde{p}^{i}\in \Delta ^{i}}\frac{g^{i}(\widetilde{p}^{i},p^{-i})%
}{f^{i}(\widetilde{p}^{i},p^{-i})}=\max_{s^{i}\in S^{i}}\frac{%
g^{i}(s^{i},p^{-i})}{f^{i}(s^{i},p^{-i})}
\end{equation*}%
which is the desired result.
\end{proof}

Again, the min-max optimization problem (\ref{minmaxloomis})
is a particular instance of the $\mrf$ problem (\ref%
{defp}) and so can be solved via the hierarchy of semidefinite relaxations (%
\ref{lower}). Notice that in (\ref{minmaxloomis}) one has to minimize the supremum of rational functions (in contrast to the supremum of polynomials in (\ref{iff})).

\subsection{Finite absorbing games}

This subclass of stochastic games has been introduced by Kohlberg \cite%
{Kohlberg}. The following formulas are established in \cite{Laraki}. It shows that absorbing games could be reduced to Loomis games. An $N$-player finite absorbing game is defined as follows.

As above, there are $N$ finite sets $(S^{1},...,S^{N}).$ There are two functions $g^{i}:\mathbf{S}\rightarrow \mathbb{R}$ and $f^{i}:\mathbf{S}\rightarrow \mathbb{R}$ associated to each player $i\in \{1,...,N\}$ and a
probability transition function $q:\mathbf{S}\rightarrow \mathbb{[}0,1%
\mathbb{]}$.

The game is played in discrete time as follows. At each stage $t=1,2,...$, if the game has not been absorbed before that day, each player $i$ chooses (simultaneously) at
random an action $s_{t}^{i}\in S^{i}$. If the profile $s_t=(s_{t}^{1},...,s_{t}^{N})$ is chosen, then:

(i) the payoff of player $i$ is $g^{i}(s_t)$ at stage $t$.

(ii) with probability $1-q(s_t) $ the game
is terminated (absorbed) and each player $i$ gets at every stage $s>t$ the
payoff $f^{i}(s_t),$ and

(iii) with probability $q(s_t)$ the game
continues (the situation is repeated at stage $t+1$).

Consider the $\lambda $-discounted game $G\left( \lambda \right) $ ($%
0<\lambda <1$). If the payoff of player $i$ at stage $t$ is $r^{i}(t)$ then
its $\lambda $-discounted payoff in the game is $\sum_{t=1}^{\infty }\lambda
(1-\lambda )^{t-1}r^{i}(t)$. Hence, a player is optimizing his expected $%
\lambda $-discounted payoff.

Let $\widetilde{g}^{i}=g^{i}\times q$ and $\widetilde{f}^{i}=f^{i}\times
(1-q)$ and extend $\widetilde{g}^{i},$ $\widetilde{f}^{i}$ and $q$
multilinearly to $\mathbf{\Delta }$ (as above in Nash and Loomis games).

A profile $p\in \Delta $ is a stationary equilibrium of the
absorbing game if (1) each player $i$ plays iid at each stage $t$ the mixed strategy $p^i$ until the game is absorbed and (2) this is optimal for him in the discounted absorbing game if the other players do not deviate.
\begin{lemma}
\label{coroabsorb} A profile $p_{0}\in \Delta $ is a stationary equilibrium of the
absorbing game if and only if it is a Nash equilibrium of the Loomis game with
payoff functions $p\rightarrow \frac{\lambda \widetilde{g}^{i}(p)+(1-\lambda
)\widetilde{f}^{i}(p)}{\lambda q(p)+(1-q(p))},$ $i=1,...,N$.
\end{lemma}

\begin{proof}
See Laraki \cite{Laraki}.
\end{proof}

As shown in \cite{Laraki}, the min-max of a discounted absorbing game
satisfies:
\begin{equation*}
\underline{v}_{i}=\min_{p^{-i}\in \mathbf{\Delta }^{-i}}\max_{s^{i}\in S^{i}}%
\frac{\lambda \widetilde{g}^{i}(s^{i},p^{-i})+(1-\lambda )\widetilde{f}%
^{i}(s^{i},p^{-i})}{\lambda q(s^{i},p^{-i})+(1-q(s^{i},p^{-i}))}.
\end{equation*}
Hence solving a finite absorbing game is equivalent to solving a Loomis game
(hence a particular instance of the $\mrf$ problem (\ref{defp})) which again can be solved via the hierarchy of semidefinite relaxations (\ref{lower}). Again one has to minimize the supremum
of rational functions.

\section{Zero-sum polynomial games}

Let $\mathbf{K}_{1}\subset\R^{n_1}$ and $\mathbf{K}_{2}\subset \mathbb{R}^{n_2}$ be two basic and
closed semi-algebraic sets (not necessarily with same dimension):
\begin{eqnarray}
\mathbf{K}_{1}:= &\{x\in &\mathbb{R}^{n_1}:g_{j}(x)\geq 0,\quad j=1,\ldots
,m_{1}\}  \label{setdelta2} \\
\mathbf{K}_{2}:= &\{x\in &\mathbb{R}^{n_2}:h_{k}(x)\geq 0,\quad k=1,\ldots
,m_{2}\}
\end{eqnarray}%
for some polynomials $(g_{j})\subset \mathbb{R}[x_1,\ldots x_{n_1}]$
and $(h_{k})\subset \mathbb{R}[x_1,\ldots x_{n_2}]$.

Let $P(\mathbf{K}_{i})$ be the set of Borel probability measures on $\mathbf{%
K}_{i}$, $i=1,2$, and consider the following min-max problem:
\begin{equation}
\P:\quad J^{\ast }\,=\,\min_{\mu \in P(\mathbf{K}_{1})}\,\max_{\nu
\in P(\mathbf{K}_{2})}\,\displaystyle\int_{\K_2} \int_{\K_1} p(x,z)\,d\mu (x)\,d\nu (z)
\label{defp1}
\end{equation}%
for some polynomial $p\in \mathbb{R}[x,z]$.

If $\mathbf{K}_{1}$ and $\mathbf{K}_{2}$ are compact, it is well-known that
\begin{eqnarray}
J^{\ast } &=&\min_{\mu \in P(\mathbf{K}_{1})}\,\max_{\nu \in P(\mathbf{K}%
_{2})}\,\displaystyle\int_{\K_2} \int_{\K_1} p(x,z)\,d\mu (x)\,d\nu (z)  \label{defp2} \\
&=&\max_{\nu \in P(\mathbf{K}_{2})}\,\min_{\mu \in P(\mathbf{K}_{1})}\,%
\displaystyle\int_{\K_1} \int_{\K_2} p(x,z)\,d\nu (z)\,d\mu (x),  \notag
\end{eqnarray}%
that is, there exist $\mu ^{\ast }\in P(\mathbf{K}_{1})$ and $\nu ^{\ast
}\in P(\mathbf{K}_{2})$ such that:
\begin{equation}
J^{\ast }\,=\,\displaystyle\int_{\K_2} \int_{\K_1} p(x,z)\,d\mu ^{\ast }(x)\,d\nu ^{\ast
}(z).  \label{attained}
\end{equation}%
The probability measures $\mu ^{\ast }$ and $\nu ^{\ast }$ are the optimal
strategies of players $1$ and $2$ respectively.

For the univariate case $n=1$, Parrilo \cite{Parrilo} showed that
$J^*$ is the optimal value of a {\it single} semidefinite program, namely
the semidefinite program (7) in \cite[p. 2858]{Parrilo}, and mentioned how to extract optimal strategies since there exist optimal strategies $(\mu^*,\nu^*)$ with finite support.
In \cite{Parrilo} the author mentions that extension to the multivariate case is possible.
We provide below such an extension which, in view of the proof of its validity
given below, is non trivial. The price to pay for this extension is to replace a single
semidefinite program with a hierarchy of semidefinite programs of increasing size.
But contrary to the polynomial optimization case in  e.g. \cite{lasserre1}, proving convergence of this hierarchy is more delicate because one has
(simultaneously in the same SDP) moment matrices of increasing size
and an s.o.s.-representation of some polynomial in Putinar's form (\ref{putrep}) with increasing degree bounds for the s.o.s. weights. In particular, the convergence is not monotone anymore.
When we do $n=1$ in this extension,
one retrieves the original result of Parrilo \cite{Parrilo}, i.e.,
the first semidefinite program in the hierarchy (\ref{sdp11}) reduces to
(7) in \cite[p. 2858]{Parrilo} and provides us with the exact desired value.

\subsection*{Semidefinite relaxations for solving $\P$}

With $p\in\mathbb{R}[x,z]$ as in (\ref{defp}), write
\begin{eqnarray}  \label{decomp}
p(x,z)&=&\sum_{\alpha\in\mathbb{N}^{n_2}}p_\alpha(x)\,z^\alpha\qquad%
\mbox{with} \\
p_\alpha(x)&=&\sum_{\beta\in\mathbb{N}^{n_1}}p_{\alpha\beta}\,x^\beta,\qquad
\vert\alpha\vert\,\leq\,d_z  \notag
\end{eqnarray}
where $d_z$ is the total degree of $p$ when seen as polynomial in $\mathbb{R}%
[z]$. So, let $p_{\alpha\beta}:=0$ for every $\beta\in\mathbb{N}^{n_1}$
whenever $\vert\alpha\vert > d_z$.\newline

Let $r_j:=\lceil\mathrm{deg}\,g_j/2\rceil$, for every $j=1,\ldots,m_1$, and
consider the following semidefinite program:

\begin{equation}
\left\{
\begin{array}{ll}
\displaystyle\min_{\mathbf{y},\lambda ,Z^{k}} & \lambda \\
\mathrm{s.t.} & \lambda \,I_{\alpha =0}-\displaystyle\sum_{\beta \in \mathbb{%
N}^{n_{1}}}p_{\alpha \beta }\,y_{\beta }\,=\,\langle Z^{0},B_{\alpha
}\rangle +\displaystyle\sum_{k=1}^{m_{2}}\langle Z^{k},B_{\alpha
}^{h_{k}}\rangle ,\quad |\alpha |\leq 2d \\
&  \\
& M_{d}(\mathbf{y})\,\succeq 0 \\
& M_{d-r_{j}}(g_{j},\mathbf{y})\,\succeq 0,\quad j=1,\ldots ,m_{1} \\
& y_{0}\,=1 \\
&  \\
& Z^{k}\,\succeq 0,\quad k=0,1,\ldots m_{2},%
\end{array}%
\right.  \label{sdp1}
\end{equation}%
where $\mathbf{y}$ is a finite sequence indexed in the canonical basis $%
(x^{\alpha })$ of $\mathbb{R}[x]_{2d}$. Denote by $\lambda _{d}^{\ast }$ its
optimal value. In fact, with $h_{0}\equiv 1$ and $p_\y\in
\mathbb{R}[z]$ defined by:
\begin{equation}
z\,\mapsto \,p_\y(z)\,:=\sum_{\alpha \in \mathbb{N}^{n_{2}}}\left(
\sum_{\beta \in \mathbb{N}^{n_{1}}}p_{\alpha \beta }\,y_{\beta }\right)
\,z^{\alpha },  \label{py}
\end{equation}%
the semidefinite program (\ref{sdp1}) has the equivalent formulation:
\begin{equation}
\left\{
\begin{array}{ll}
\displaystyle\min_{\mathbf{y},\lambda ,\sigma _{k}} & \lambda \\
\mathrm{s.t.} & \lambda \,-\,p_\y(\cdot )\,=\,\displaystyle%
\sum_{k=0}^{m_{2}}\sigma _{k}\,h_{k} \\
&  \\
& M_{d}(\mathbf{y})\,\succeq 0 \\
& M_{d-r_{j}}(g_{j},\mathbf{y})\,\succeq 0,\quad j=1,\ldots ,m_{1} \\
& y_{0}\,=1 \\
&  \\
& \sigma _{k}\in \Sigma \lbrack z];:\mathrm{deg}\,\sigma _{k}+\mathrm{deg}%
\,h_{k}\leq 2d,\quad k=0,1,\ldots ,m_{2},%
\end{array}%
\right.  \label{sdp11}
\end{equation}%
where the first constraint should be understood as an equality of polynomials. Observe that for any admissible solution $(\mathbf{y},\lambda )$ and $p_\y$ as in (\ref{py}),
\begin{equation}
\lambda \,\geq \max_{z}\{p_\y(z)\,:\,z\in \mathbf{K}_{2}\}.
\label{test2}
\end{equation}

Similarly, with $p$ as in (\ref{defp}), write
\begin{eqnarray}  \label{decomp2}
p(x,z)&=&\sum_{\alpha\in\mathbb{N}^{n_1}}\hat{p}_\alpha(z)\,x^\alpha\qquad%
\mbox{with} \\
\hat{p}_\alpha(z)&=&\sum_{\beta\in\mathbb{N}^{n_2}}\hat{p}%
_{\alpha\beta}\,z^\beta,\qquad \vert\alpha\vert\,\leq\,d_x  \notag
\end{eqnarray}
where $d_x$ is the total degree of $p$ when seen as polynomial in $\mathbb{R}%
[x]$. So, let $\hat{p}_{\alpha\beta}:=0$ for every $\beta\in\mathbb{N}^{n_2}$
whenever $\vert\alpha\vert > d_x$.\newline

Let $l_k:=\lceil\mathrm{deg}\,h_k/2\rceil$, for every $k=1,\ldots,m_2$, and
with
\begin{equation}  \label{newpy}
x\,\mapsto\,\hat{p}_\y(x)\,:=\sum_{\alpha\in\mathbb{N}%
^{n_1}}\left(\sum_{\beta\in\mathbb{N}^{n_2}}\hat{p}_{\alpha\beta}\,y_\beta%
\right)\,x^\alpha,
\end{equation}
consider the following semidefinite program (with $g_0\equiv 1$):

\begin{equation}  \label{sdp2}
\left\{%
\begin{array}{ll}
\displaystyle\max_{\mathbf{y},\gamma,\sigma_j} & \gamma \\
\mathrm{s.t.} & \hat{p}_\y(\cdot)-\gamma \,=\,\displaystyle%
\sum_{j=0}^{m_1}\sigma_j\,g_j \\
&  \\
& M_d(\mathbf{y})\,\succeq0 \\
& M_{d-l_k}(h_k,\mathbf{y})\,\succeq0,\quad k=1,\ldots,m_2 \\
& y_0\,=1 \\
&  \\
& \sigma_j\in\Sigma[x];\:\mathrm{deg}\,\sigma_j+\mathrm{deg}\,g_j\leq
2d,\quad j=0,1,\ldots, m_1.%
\end{array}%
\right.
\end{equation}
where $\mathbf{y}$ is a finite sequence indexed in the canonical basis $%
(z^\alpha)$ of $\mathbb{R}[z]_{2d}$. Denote by $\gamma^*_d$ its optimal
value. In fact, (\ref{sdp2}) is the dual of the semidefinite program (\ref%
{sdp1}).\newline

Observe that for any admissible solution $(\mathbf{y},\gamma )$ and $\hat{p}_\y$ as in (\ref{newpy}),
\begin{equation}
\gamma \,\leq \,\min_{x}\{\hat{p}_\y(x)\,:\,x\in \mathbf{K}_{1}\}.
\label{test}
\end{equation}

\begin{theorem}
\label{thmain2} Let $\P$ be the min-max problem defined in (\ref%
{defp}) and assume that both $\K_1$ and $\K_2$ are compact
and satisfy Putinar's property (see Definition \ref{defput}).
Let $\lambda _{d}^{\ast }$ and $%
\gamma _{d}^{\ast }$ be the optimal values of the semidefinite program (\ref%
{sdp11}) and (\ref{sdp2}) respectively. Then $\lambda _{d}^{\ast
}\rightarrow J^{\ast }$ and $\gamma _{d}^{\ast }\rightarrow J^{\ast }$ as $%
d\rightarrow \infty $.
\end{theorem}

We also have a test to detect whether finite convergence has occurred.

\begin{theorem}
\label{th2} Let $\P$ be the min-max problem defined in (\ref{defp}).

Let $\lambda^*_d$ be the optimal value of the semidefinite
program (\ref{sdp11}), and suppose that with $r:=\max_{j=1,\ldots,m_1}r_j$,
the condition
\begin{equation}  \label{th2-1}
\mathrm{rank}\,M_{d-r}(\mathbf{y})\,=\,\mathrm{rank}\,M_{d}(\mathbf{y}%
)\quad(=:\,s_1)
\end{equation}
holds at an optimal solution $(\mathbf{y},\lambda,\sigma_k)$ of (\ref{sdp11}).

Let $\gamma^*_t$ be the optimal value of the semidefinite
program (\ref{sdp2}), and suppose that with $r:=\max_{k=1,\ldots,m_2}l_k$,
the condition
\begin{equation}  \label{th22-1}
\mathrm{rank}\,M_{t-r}(\mathbf{y'})\,=\,\mathrm{rank}\,M_{t}(\mathbf{y'})\quad(=:\,s_2)
\end{equation}
holds at an optimal solution $(\mathbf{\y'},\gamma,\sigma_j)$ of (\ref{sdp2}).

If $\lambda^*_d=\gamma^*_t$ then $\lambda^*_g=\gamma _{t}^{\ast }=J^{\ast }$
and an optimal strategy for player 1 (resp. player 2) is
a probability measure $\mu^*$ (resp. $\nu^*$) supported on $s_{1}$ points of $\mathbf{K}_{1}$
(resp. $s_{2}$ points of $\mathbf{K}_{2}$).
\end{theorem}

For a proof the reader is referred to \S \ref{proof2}.

\begin{remark}
\label{univariate} {\rm In the univariate case, when $\mathbf{K}_{1},%
\mathbf{K}_{2}$ are (not necessarily bounded) intervals of the real line,
the optimal value $J^{\ast }=\lambda _{d}^{\ast }$ (resp. $J^{\ast }=\gamma
_{d}^{\ast }$) is obtained by solving the single semidefinite program (\ref%
{sdp11}) (resp. (\ref{sdp2})) with $d=d_{0}$, which is equivalent to (7) in Parrilo \cite[p. 2858]{Parrilo}. }
\end{remark}

\section{Zero-sum polynomial absorbing games}

As in the previous section, consider two compact basic semi-algebraic sets $%
\mathbf{K}_{1}\subset \mathbb{R}^{n_{1}}$, $\mathbf{K}_{2}\subset \mathbb{R}%
^{n_{2}}$ and polynomials $g,$ $f$ and $q:\mathbf{K}_{1}\times \mathbf{K}%
_{2}\rightarrow \left[ 0,1\right] $. Recall that $P(\mathbf{K}_{1})$ (resp. $%
P(\mathbf{K}_{2})$) denotes the set of probability measures on $\mathbf{K}%
_{1}$ (resp. $\mathbf{K}_{2}$). The absorbing game is played in discrete
time as follows. At stage $t=1,2,...$ player 1 chooses at random $x_{t}\in
\mathbf{K}_{1}$ (using some mixed action $\mu _{t}\in P\left( \mathbf{K}%
_{1}\right) )$ and, simultaneously, Player 2 chooses at random $y_{t}\in
\mathbf{K}_{2}$ (using some mixed action $\nu _{t}\in P\left( \mathbf{K}%
_{2}\right) $).

(i) player 1\ receives $g\left( x_{t},y_{t}\right) $ at stage $t$;

(ii) with probability $1-q\left( x_{t},y_{t}\right) $ the game is absorbed
and player 1 receives $f\left( x_{t},y_{t}\right) $ in all stages $s>t$;

and

(iii) with probability $q\left( x_{t},y_{t}\right) $ the game continues (the
situation is repeated at step $t+1$).

If the stream of payoffs is $r(t)$, $t=1,2,...,$ the $\lambda $%
-discounted-payoff of the game is $\sum_{t=1}^{\infty }\lambda (1-\lambda
)^{t-1}r(t)$.

Let $\widetilde{g}=g\times q$ and $\widetilde{f}=f\times (1-q)$ and extend $%
\widetilde{g},$ $\widetilde{f}$ and $q$ multilinearly to $P\left( \mathbf{K}%
_{1}\right) \times P\left( \mathbf{K}_{2}\right) $.

Player 1 maximizes the expected discounted-payoff and player 2 minimizes
that payoff. Using an extension of the Shapley operator \cite{Shapley} one
can deduce that the game has a value $v\left( \lambda \right) $ that
uniquely satisfies,
\begin{eqnarray*}
v\left( \lambda \right) &=&\max_{\mu \in P(\mathbf{K}_{1})}\min_{\nu \in P(%
\mathbf{K}_{2})}\int_{\Theta }\left( \lambda \widetilde{g}+(1-\lambda
)v(\lambda )p+(1-\lambda )\widetilde{f}\right) \,d\mu \otimes \nu \\
&=&\min_{\nu \in P(\mathbf{K}_{2})}\max_{\mu \in P(\mathbf{K}%
_{1})}\int_{\Theta }\left( \lambda \widetilde{g}+(1-\lambda )v(\lambda
)p+(1-\lambda )\widetilde{f}\right) \,d\mu \otimes \nu
\end{eqnarray*}%
with $\Theta :=\mathbf{K}_{1}\times \mathbf{K}_{2}$. As in the finite case, it may be shown
\cite{Laraki} that the problem may be reduced to a zero-sum Loomis game, that is:

\begin{equation}
v\left( \lambda \right) \,=\,\max_{\mu \in P(\mathbf{K}_{1})}\min_{\nu \in P(%
\mathbf{K}_{2})}\frac{\int_{\Theta }P\,d\mu \otimes \nu }{\int_{\Theta
}Q\,d\mu \otimes \nu }\,=\,\min_{\nu \in P(\mathbf{K}_{2})}\max_{\mu \in P(%
\mathbf{K}_{1})}\frac{\int_{\Theta }P\,d\mu \otimes \nu }{\int_{\Theta
}Q\,d\mu \otimes \nu }
\end{equation}

where%
\begin{eqnarray*}
(x,y)\mapsto P(x,y) &:=&\lambda \widetilde{g}(x,y)+(1-\lambda )\widetilde{f}%
(x,y) \\
(x,y)\mapsto Q(x,y) &:=&\lambda q(x,y)+1-q(x,y)
\end{eqnarray*}

Or equivalently, as it was originally presented by Loomis \cite{Loomis}, $v\left( \lambda \right) $ is the unique real $t$ such that%
\begin{eqnarray*}
0 &=&\max_{\mu \in P(\mathbf{K}_{1})}\min_{\nu \in P(\mathbf{K}_{2})}\left[
\int_{\Theta }(P(x,y)-t\,Q(x,y))\,d\mu (x)d\nu (y)\right] \\
&=&\min_{\nu \in P(\mathbf{K}_{2})}\min_{\nu \in P(\mathbf{K}_{1})}\left[
\int_{\Theta }(P(x,y)-t\,Q(x,y))\,d\mu (x)d\nu (y)\right] .
\end{eqnarray*}

Actually, the function $s:\mathbb{R}\rightarrow \mathbb{R}$ defined by:
\begin{equation*}
t\rightarrow s(t):=\max_{\mu \in P(\mathbf{K}_{1})}\min_{\nu \in P(\mathbf{K}%
_{2})}\left[ \int_{\Theta }(P(x,y)-tQ(x,y))\,d\mu (x)d\nu (y)\right]
\end{equation*}%
is continuous, strictly decreasing from $+\infty $ to $-\infty $ as
$t$ increases in $(-\infty,+\infty)$.

In the univariate case, if $\mathbf{K}_{1}$ and $\mathbf{K}_{2}$ are both
real intervals (not necessarily compact), then evaluating $s(t)$ for some
fixed $t$ can be done by solving a \textit{single} semidefinite program; see
Remark \ref{univariate}. Therefore, in this case, one may approximate the
optimal value $s^{\ast }\,(=s(t^{\ast }))$ of the game by binary search on $t$
and so, the problem can be solved in a polynomial time. This extends Shah
and Parrilo \cite{Shah}.

\section{Conclusion}

We have proposed a common methodology to approximate the optimal value of
games in two different contexts. The first algorithm, intended to compute
(or approximate) Nash equilibria in mixed strategies for static finite games
or dynamic absorbing games, is based on a hierarchy of semidefinite programs
to approximate the supremum of finitely many rational functions on a compact
basic semi-algebraic set.  Actually this latter formulation is also of self-interest
in optimization.
The second algorithm, intended to approximate the
optimal value of polynomial games whose action sets are compact basic
semi-algebraic sets, is also based on a hierarchy of semidefinite programs.
Not surprisingly, as the latter algorithm comes from a min-max problem over
sets of measures, it is a subtle combination of moment and s.o.s. constraints
whereas in polynomial optimization it is entirely formulated either with
moments (primal formulation) or with s.o.s. (dual formulation). Hence the above
methodology illustrates the power of the combined moment-s.o.s. approach.
A natural open question arises: how to
adapt the second algorithm to compute Nash equilibria of a non-zero-sum
polynomial game?

\section{Appendix}
\label{proof2}

\subsection{Proof of Theorem \protect\ref{thmain2}}

We first need the following partial result.

\begin{lemma}
\label{aux} Let $(\mathbf{y}^d)_d$ be a sequence of admissible solutions of
the semidefinite program (\ref{sdp1}). Then there exists $\hat{\mathbf{%
\mathbf{y}}}\in\mathbb{R}^\infty $ and a subsequence $(d_i)$ such that $%
\mathbf{\mathbf{y}}^{d_i}\to \hat{\mathbf{\mathbf{y}}}$ pointwise as $%
i\to\infty$, that is,
\begin{equation}  \label{pointwise}
\lim_{i\to\infty}\,y^{d_i}_\alpha\,=\,\hat{y}_\alpha,\qquad\forall \alpha\in%
\mathbb{N}^n.
\end{equation}
\end{lemma}

The proof is omitted because it is exactly along the same lines as the proof
of Theorem \ref{thmain} as among the constraints of the feasible set, one
has
\begin{equation*}
y^d_0=1,\quad M_{d}(\mathbf{y}^d)\succeq0,\quad M_d(g_j,\,\mathbf{y}%
^d)\succeq0,\:j=1,\ldots,m_1.
\end{equation*}

\subsubsection*{Proof of Theorem \protect\ref{thmain2}}

Let $\mu ^{\ast }\in P(\mathbf{K}_{1}),\nu ^{\ast }\in P(\mathbf{K}_{2})$ be
optimal strategies of player 1 and player 2 respectively, and let $\mathbf{y}%
^{\ast }=(y_{\alpha }^{\ast })$ be the sequence of moments of $\mu ^{\ast }$
(well-defined because $\mathbf{K}_{1}$ is compact). Then
\begin{eqnarray*}
J^{\ast } &=&\max_{\nu \in P(\mathbf{K}_{2})}\int_{\K_2} \left( \int_{\K_1} p(x,z)d\mu
^{\ast }(x)\right) \,d\nu (z) \\
&=&\max_{\nu \in P(\mathbf{K}_{2})}\int_{\K_2} \sum_{\alpha \in \mathbb{N}%
^{n_2}}\left( \sum_{\beta \in \mathbb{N}^{n_1}}p_{\alpha \beta }\int_{\K_1}x^{\beta
}\,d\mu ^{\ast }(x)\right) \,z^{\alpha }d\nu (z) \\
&=&\max_{\nu \in P(\mathbf{K}_{2})}\int_{\K_2} \sum_{\alpha \in \mathbb{N}%
^{n_2}}\left( \sum_{\beta \in \mathbb{N}^{n_1}}p_{\alpha \beta }y_{\alpha \beta
}^{\ast }\right) \,z^{\alpha }d\nu (z) \\
&=&\max_{\nu \in P(\mathbf{K}_{2})}\int_{\K_2} p_{\y^{\ast}}(z)\,d\nu (z) \\
&=&\max_{z}\:\{p_{\y^{\ast }}(z)\,:\,z\in \mathbf{K}_{2}\} \\
&=&\min_{\lambda ,\sigma _{k}}\{\lambda :\lambda -p_{\y^{\ast}}(\cdot
)\,=\,\sigma _{0}+\sum_{k=1}^{m_{2}}\sigma _{k}\,h_{k};\quad (\sigma
_{j})_{j=0}^{m_{2}}\subset \Sigma \lbrack z]\}
\end{eqnarray*}%
with $z\mapsto \,p_{\y^{\ast}}(z)$ defined in (\ref{py}). Therefore,
with $\epsilon>0$ fixed arbitrary,
\begin{equation}  \label{putepsilon}
J^*-p_{\y^*}(\cdot)+\epsilon\,=\,\sigma_0^\epsilon+\sum_{k=1}^{m_2}%
\sigma_k^\epsilon\,h_k,
\end{equation}
for some polynomials $(\sigma_k^\epsilon)\subset\Sigma[z]$ of degree at most
$2d^1_\epsilon$. So $(\mathbf{y}^*,J^*+\epsilon,\sigma^\epsilon_k)$ is an
admissible solution for the semidefinite program (\ref{sdp11}) whenever $%
d\geq\max_jr_j$ and $d\geq d^1_\epsilon+\max_k l_k$, because
\begin{equation}  \label{deg}
2d\geq\mathrm{deg}\,\sigma^\epsilon_0\,;\quad 2d\geq\mathrm{deg}%
\,\sigma^\epsilon_k+\mathrm{deg}\,h_k,\quad k=1,\ldots,m_2.
\end{equation}
Therefore,
\begin{equation}  \label{value}
\lambda^*_d\,\leq\,J^*+\epsilon,\qquad \forall d\geq\, \tilde{d}%
^1_\epsilon:=\max\left[\max_jr_j,\,d^1_\epsilon+\max_kl_k\,\right].
\end{equation}

Now, let $(\mathbf{y}^{d},\lambda _{d})$ be an admissible solution of the
semidefinite program (\ref{sdp11}) with value $\lambda _{d}\leq \lambda
_{d}^{\ast }+1/d$. By Lemma \ref{aux}, there exists $\hat{\mathbf{y}}\in
\mathbb{R}^{\infty }$ and a subsequence $(d_{i})$ such that $\mathbf{y}%
^{d_{i}}\rightarrow \hat{\mathbf{y}}$ pointwise, that is, (\ref{pointwise})
holds. But then, invoking (\ref{pointwise}) yields
\begin{equation*}
M_{d}(\hat{\mathbf{y}})\succeq 0\quad \mbox{and}\quad M_{d}(g_{j},\hat{%
\mathbf{y}})\succeq 0,\quad \forall j=1,\ldots ,m_{1};\quad d=0,1,\dots
\end{equation*}%
By Theorem \ref{thput}, there exists $\hat{\mu}\in P(\mathbf{K}_{1})$ such
that
\begin{equation*}
\hat{y}_{\alpha }\,=\,\int_{\K_1} x^{\alpha }\,d\hat{\mu},\qquad \forall \alpha \in
\mathbb{N}^{n}.
\end{equation*}%
On the other hand,
\begin{eqnarray*}
J^{\ast } &\leq &\max_{\nu \in P(\mathbf{K}_{2})}\int_{\K_2} \left( \int_{\K_1} p(x,z)d%
\hat{\mu}(x)\right) \,d\nu (z) \\
&=&\max_{z}\,\{p_{\hat{\y}}(z)\,:\,z\in \mathbf{K}_{2}\} \\
&=&\min \{\lambda :\lambda -p_{\hat{\y}}(\cdot )\,=\,\sigma
_{0}+\sum_{k=1}^{m_{2}}\sigma _{k}\,h_{k};\quad (\sigma
_{j})_{j=0}^{m_{2}}\subset \Sigma \lbrack z]\}
\end{eqnarray*}%
with
\begin{equation*}
z\,\mapsto \,p_{\hat{\y}}(z)\,:=\sum_{\alpha \in \mathbb{N}^{n}}\left(
\sum_{\beta \in \mathbb{N}^{n}}p_{\alpha \beta }\,\hat{y}_{\beta }\right)
\,z^{\alpha }.
\end{equation*}%
Next, let $\rho :=\max_{z\in \mathbf{K}_{2}}p_{\hat{\y}}(z)$ (hence $%
\rho\geq J^*$), and
consider the polynomial
\begin{equation*}
z\,\mapsto \,p_{\y^{d}}(z)\,:=\sum_{\alpha \in \mathbb{N}^{n}}\left(
\sum_{\beta \in \mathbb{N}^{n}}p_{\alpha \beta }\,y_{\beta }^{d}\right)
\,z^{\alpha }.
\end{equation*}%
It has same degree as $p_{\hat{\y}}$, and by (\ref{pointwise}%
), $\Vert p_{\hat{\y}}(\cdot )-p_{\y^{d_i}}(\cdot )\Vert
\rightarrow 0$ as $i\rightarrow \infty $.

Hence, $\max_{z\in \mathbf{K}_{2}}\,p_{\y^{d_i}}(z)\rightarrow \rho $
as $i\rightarrow \infty $, and by construction of the semidefinite program (%
\ref{sdp11}), $\lambda _{d_{i}}^{\ast }\geq \max_{z\in \mathbf{K}_{2}}\,
p_{\y^{d_i}}(z)$.

Therefore, $\lambda^*_{d_i}\geq \rho-\epsilon$ for all sufficiently large $i$
(say $d_i\geq d^2_\epsilon$) and so, $\lambda^*_{d_i}\geq J^*-\epsilon$ for
all $d_i\geq d^2_\epsilon$. This combined with $\lambda^*_{d_i}\leq
J^*+\epsilon$ for all $d_i\geq \tilde{d}^1_\epsilon$, yields the desired
result that $\lim_{i\to\infty}\lambda^*_{d_i}=J^*$ because $\epsilon>0$
fixed was arbitrary;

Finally, as the converging subsequence $(r_{i})$ was arbitrary, we get that
the entire sequence $(\lambda _{d}^{\ast })$ converges to $J^{\ast }$.
The convergence $\gamma^*_d\to J^{\ast}$ is proved with similar arguments. $\qed$

\subsection{Proof of Theorem \protect\ref{th2}}

By the flat extension theorem of Curto and Fialkow (see e.g. \cite{lasserre2}),
$\y$ has a representing $s_1$-atomic probability measure $\mu^*$ supported on $\K_1$
and similarly, $\y'$ has a representing $s_2$-atomic probability measure $\nu^*$ supported on $\K_2$.
But then from the proof of Theorem \ref{thmain2},
\begin{eqnarray*}
J^\ast&\leq &\max_{\psi\in P(\K_2)}\int_{\K_2}\left(\int_{\K_1} p(x,z)d\mu^*(x) \right)d\psi(z)\\
 &=&\max_{\psi\in P(\K_2)}\int_{\K_2} p_\y(z)d\psi(z)=\max_{z}\,\{\,p_\y(z)\,:\:z\in\K_2\}\leq \lambda^*_d.\\
J^\ast&\geq& \min_{\psi\in P(\K_1)}\int_{\K_1}\left(\int_{\K_2} p(x,z)d\nu^*(z)\right)d\psi(x)\\
&=&\min_{\psi\in P(\K_1)}\int_{\K_1}\hat{p}_{\y'}(x)d\psi(x)=\min_{x}\,\{\,\hat{p}_{\y'}(x)\,:\:x\in\K_1\}\geq\gamma^*_t,\end{eqnarray*}
and so as $\lambda^*_d=\gamma^*_t$ one has $J^*=\lambda^*_d=\gamma^*_t$. This in turn implies that $\mu^*$ (resp. $\nu^*$) is an optimal strategy for player $1$ (resp. player $2$). $\qed$
\section*{Acknowledgments}

This work was supported by french ANR-grant NT$05-3-41612$.

blabla
\end{document}